\documentclass[hidelinks,12pt,a4paper,reqno]{amsart}% % onecolumn (second format)
\usepackage{graphicx}
\usepackage[all]{xy}
\usepackage{amssymb}
\usepackage{amsmath}
\usepackage[utf8]{inputenc}
\usepackage{amsthm}

\newtheorem{proposition}{Proposition}
\newtheorem{remark}{Remark}
\newtheorem{definition}{Definition}
\newtheorem{theorem}{Theorem}

\usepackage[backref=page]{hyperref}
\renewcommand*{\backref}[1]{}
\renewcommand*{\backrefalt}[4]{%
    \ifcase #1 (Not cited.)%
    \or        (Cited on page~#2.)%
    \else      (Cited on pages~#2.)%
    \fi}
\newcommand{\N}{\mathbb{N}}

\newcommand{\R}{\mathbb{R}}

\newcommand{\Ker}{\mathrm{Ker}}

\title[Semi-biproducts of monoids]{Semi-biproducts of monoids}

\author{Nelson Martins-Ferreira}
\address[Nelson Martins-Ferreira]{Instituto Politécnico de Leiria, Leiria, Portugal}
\thanks{ }
\email{martins.ferreira@ipleiria.pt}

\begin{document}

\begin{abstract}
It is shown that the category of \emph{semi-biproducts} of monoids is equivalent to the category of \emph{pseudo-actions}. A semi-biproduct of monoids is a new notion, obtained through generalizing a biproduct of commutative monoids. By dropping commutativity and requiring some of the homomorphisms in the biproduct diagram to be merely identity-preserving maps, we obtain a semi-biproduct. A pseudo-action is a new notion as well. It consists of three ingredients: a pre-action, a factor system and a correction system. In the category of groups all correction systems are trivial. This is perhaps the reason why this notion, to the author's best knowledge, has never been considered before.
\end{abstract}


% ams classification
% 18G50, 20M10, 20M32

\maketitle

\section{Introduction}

In this paper, we introduce and study the notion of a semi-biproduct of monoids.
%The purpose of this paper is to introduce the notion of a semi-biproduct and to investigate it in the category of monoids.
From a categorical point of view a biproduct is simultaneously a product and a coproduct (see \cite{MacLane}, p.194). Furthermore, when there is a coincidence between products and coproducts in a category, each hom-set carries the structure of a commutative monoid and parallel morphisms can be added. In such a context a biproduct can be defined as a diagram of the form
\begin{equation}
\label{diag: biproduct in additive cats}
\xymatrix{X\ar@<-.5ex>[r]_{k} & A\ar@<-.5ex>@{->}[l]_{q}\ar@<.5ex>[r]^{p} & B \ar@{->}@<.5ex>[l]^{s}}
\end{equation}
such that the following conditions hold:
\begin{eqnarray}
ps=1_B,\quad pk=0_{X,B},\label{eq: biproduct1}\\
qs=0_{B,X},\quad qk=1_{X},\label{eq: biproduct2}\\
kq+sp=1_A.\label{eq: biproduct3}
\end{eqnarray}
In spite of the fact that biproducts are 
% usually defined for
classically studied in
 categories whose hom-sets are abelian groups (in which case the conditions $pk=0$ and $qs=0$ are derived) \cite{MacLane}, it is clear that
%  this definition makes sense
the concept of a biproduct is meaningful
 even when each hom-set is a commutative monoid. That seems to be the most general case in which it makes sense --- but we may ask the following question: how to generalize the definition
%  so that it can be applied in the category of groups
so that it can be applied to groups, where semi-direct products are expected to appear?
The present work gives an answer to this question
%This work answers the question
 and  it goes beyond groups. It provides a generalization to the so called \emph{Schreier extensions} of monoids. However, we answer the question at the expense of allowing consideration of functions between underlying sets of groups/monoids, and hence stepping out from the context of the category of groups/monoids. See  \cite{BournJanelidze} for a purely categorical approach to semi-direct products of groups.

A Schreier extension of monoids (see \cite{NMF et all}) can be seen as a sequence of monoid homomorphisms $$\xymatrix{X\ar[r]^{k}& A \ar[r]^{p} & B,}$$ in which $k=\ker(p)$, with the property that there exist two set-theoretical functions, $s\colon{B\to A}$ and $q\colon{A\to X}$, satisfying conditions $(\ref{eq: Sch-2})$--$(\ref{eq: Sch})$ below.
% When condition $(\ref{eq: Sch})$ is dropped, then the remaining ones can be reorganized so that they become precisely the conditions  $(\ref{eq: biproduct1})$--$(\ref{eq: biproduct3})$, except that now $q$ and $s$ are not homomorphisms but simply zero-preserving maps.
 Notice the similarity between these conditions and the conditions $(\ref{eq: biproduct1})$--$(\ref{eq: biproduct3})$ defining a biproduct. In particular, note that $(\ref{eq: Sch})$ is subsumed by $(\ref{eq: biproduct1})$ and $(\ref{eq: biproduct2})$ when $q$ is a homomorphism.
Thus, the notion of semi-biproduct of monoids arrives naturally as a diagram of the from $(\ref{diag: biproduct in additive cats})$ where $p$ and $k$ are monoid homomorphisms while $q$ and $s$ are zero-preserving maps, satisfying conditions $(\ref{eq: biproduct1})$--$(\ref{eq: biproduct3})$.

Let us see in more detail the long path that has given rise to this simple definition.
In  the category of groups and group homomorphisms, every split epimorphism
\begin{equation*}
\label{diag: split epi}
\xymatrix{ A\ar@<.5ex>[r]^{p} & B \ar@{->}@<.5ex>[l]^{s}},\quad ps=1_B
\end{equation*}
induces an exact sequence
\begin{equation*}
\label{diag: exact seq}
\xymatrix{ X\ar[r]^{k} & A\ar@<.0ex>[r]^{p} & B, }
\end{equation*}
with  $X=\Ker(p)$. If fixing $X$ in additive notation (without assuming that $X$ is an abelian group) and $B$ in multiplicative notation, then,   up to isomorphism, the group $A$ is recovered as a semi-direct product of the form $X\rtimes_{\varphi}B$, with neutral element $(0,1)\in X\times B$, and group operation
\begin{equation}
(x,b)+(x',b')=(x+\varphi_b(x'),bb').
\end{equation}
Recall that (see e.g. \cite{MacLane2} and its references to previous work) $$\varphi_b\colon{X\to X}$$ is defined, for every $b\in B$, as the map $\varphi_b(x)=q(s(b)+k(x))$, with $q(a)=a-sp(a)\in X$, for every $a\in A$. Moreover, when $p$ is a surjective homomorphism, but not necessarily a split epimorphism, it is still possible to recover the group $A$ as a semi-direct product  with a factor system, $X\rtimes_{\varphi,\gamma}B$. A factor system, $\gamma\colon{B\times B\to X}$, is any map $\gamma(b,b')=q(s(b)+s(b'))=s(b)+s(b')-s(bb')$, for some chosen set-theoretical section $s\colon{B\to A}$ of $p$, i.e., such that $ps=1_B$. In this case, the group operation in $X\rtimes_{\varphi,\gamma}B$ becomes
\begin{equation}
(x,b)+(x',b')=(x+\varphi_b(x')+\gamma(b,b'),bb').
\end{equation}
When $A$ is abelian, the action $\varphi_b(x)=s(b)+k(x)-s(b)\in X$ is trivial and so, the surjective homomorphisms with abelian domain are nothing but (symmetric) factor systems.

The case of monoids is different (see \cite{DB.NMF.AM.MS.13}). It is no longer true that every surjective monoid homomorphism $p\colon{A\to B}$ can be presented as a sequence
\begin{equation*}
\label{diag: exact seq mon}
\xymatrix{ X\ar[r]^(.34){\langle 1,0\rangle} & X\rtimes_{\varphi,\gamma}B \ar@<.0ex>[r]^(.64){\pi_B} & B }
\end{equation*}
with $X=\Ker(p)$ and such that the underlying set of $A$ is in bijection with the cartesian product $X\times B$. Thus, in monoids, the transport of the structure from  $A$ to $X\times B$ is not always possible. So, there is no hope of having an isomorphism $A\cong X\rtimes_{\varphi,\gamma}B$. 
A simple example to illustrate this is obtained as
\begin{equation}\label{eg:simple eg}
\xymatrix{\{1,a\}\ar[r]^{k}&\{1,a,b\}\ar[r]^{p}&\{1,b\}}
\end{equation}
with $k$ the inclusion map, $p(1)=p(a)=1$, $p(b)=b$ and $a^2=a$, $ab=ba=b^2=b$.
% hook arrow in xymatrix @{^{(}->}

We may ask: which monoid homomorphisms exhibit a behaviour similar to the behaviour of surjective homomorphisms of groups? The answer turns out to be the so called  Schreier extensions (see \cite{Ganci, NMF et all} and references therein, see also \cite{DB.NMF.AM.MS.13}). 
%As we will see, in monoids there is a hidden ingredient, which is invisible in the category of groups, and that is perhaps the reason why it has never been considered so far (to our best knowledge). However, it is expected to clarify some otherwise obscure connections between pseudo-actions and extensions in groups, monoids, categories and groupoids.
A Schreier extension $p\colon{A\to B}$ is nothing but a monoid homomorphism  with the property that there exists a set-theoretical section $s\colon{B\to A}$ for which the function $X\times B\to A$, $(x,b)\mapsto k(x)+s(b)$ is a bijection of the underlying sets, with $k=\ker(p)$.
Equivalently, a Schreier extension can be seen as a monoid homomorphism $p\colon{A\to B}$, together with set-theoretical maps $s\colon{B\to A}$ and $q\colon{A\to X}$, with $X=\Ker(p)$ and $k=\ker(p)\colon{X\to A}$, such that
\begin{eqnarray}
ps(b)=b ,& b\in B,\label{eq: Sch-2}\\
kq(a)+sp(a)=a ,& a\in A,\\
q(k(x)+s(b))=x ,& x\in X, b\in B. \label{eq: Sch}
\end{eqnarray}

The key feature of this work is the observation that, at the expense of having to insert  a correction system, condition $(\ref{eq: Sch})$ can be discarded. The correction system controls the lack of condition $(\ref{eq: Sch})$ and will be denoted by $x^b$. It is defined as $x^b=q(k(x)+s(b))$, for every $x\in X$ and $b\in B$. While dropping condition $(\ref{eq: Sch})$, we lose the bijection between the underlying set of $A$ and the cartesian product of sets $X\times B$. However, if adding two simple conditions, namely $q(k(x))=x$ and $q(s(b))=0$, then, $A$ is bijective to a subset of the cartesian product $X\times B$. The subset consists of those pairs in  $X\times B$ that are of the form $(x^b,b)$  for some $x\in X$ and $b\in B$ (see Theorem \ref{thm: semi-biproducts}). The monoid operation is
\begin{equation}\label{eq:semi biproduct operation}
(x,b)+(x',b')=((x+b\cdot x'+(b\times b'))^{bb'},bb').
\end{equation}
In this way we see that it is possible to consider monoid extensions in which the middle object is not in bijection with the two ends, as it is the case in diagram (\ref{eg:simple eg}).

The correction system $x^b$ has to satisfy some conditions (Definition \ref{def: pseudo-action}). In the case of groups it becomes trivial, that is $x^b=x$. To see it, simply take $b'=b^{-1}$ in equation $(\ref{eq: 53})$. The notion of a correction system is inspired by the work of Leech and Wells on extending groups by monoids (\cite{Leech,Wells}). The fact that the correction system $\rho(x,b)=x^b$ is invisible in groups, hides the true nature of pseudo-actions  as a combination of structures $\rho$, $\varphi$,  $\gamma$ (Definition \ref{def: pseudo-action}) satisfying one major condition (\ref{eq: factor}). Only the particular traces of (\ref{eq: factor}) are familiar in groups. This is better explained in Remark \ref{remark}. A Schreier extension of monoids  is precisely a semi-biproduct with a trivial correction system, which explains 
%the similar behaviour between
the similarity between
 Schreier extensions of monoids and all extensions of groups. Examples of semi-biproducts of monoids that are not Schreier extensions are presented in Sections \ref{sec: eg} and \ref{sec:ext}. At this point we may observe that the sequence (\ref{eg:simple eg}), if completed with obvious section $s$ and retraction $q$, is an example of a semi-biproduct of monoids.
%As we will see, the best that we can hope for is to have $A$ (i.e. the underlying set of $A$) as a subset of the cartesian product. To do that we will need to introduce a new ingredient, a correction system, $\rho(x,b)=x^b$, so that we have $$A\cong\{(x^b,b)\}\subseteq X\times  B.$$  And yet, not every surjective monoid homomorphism  exhibits that property. 
The celebrated surjective homomorphism with the usual addition of natural numbers,
\begin{equation*}
\xymatrix{\N\times \N \ar[r]^{+} & \N,}
\end{equation*}
whose kernel is trivial, illustrates the fact that not every surjective monoid homomorphism is part of a semi-biproduct diagram.
% No matter what section $s$ to the map $+$ we choose, we would never have $(x,y)=(0,0)+s(x+y)$.

From a structural point of view, we are considering sequences of monoid homomorphisms
\begin{equation}
\xymatrix{X\ar[r]^{k} & A \ar[r]^{p} & B}
\end{equation} 
together with set-theoretical maps,  preserving the neutral element, but not necessarily preserving the monoid operation,
\begin{equation}
\xymatrix{X & A \ar[l]_{q} & B \ar[l]_{s} }
\end{equation} 
and satisfying the conditions $(\ref{eq: biproduct1})$--$(\ref{eq: biproduct3})$. When the maps $q$ and $s$ are monoid homomorphisms then this is exactly the definition of a biproduct. It is then natural to use the term \emph{semi-biproduct} when $q$ and $s$ are just zero-preserving maps.
%As expected, a semi-biproduct becomes a product as soon as the zero-preserving maps $q$ and $s$ are both morphisms. In that case it is a coproduct as well.
 Particular instances when just one of $q$ or $s$ is a morphism are  worthwile studying, but we shall no longer expect a product or a coproduct to appear. In general it does not make sense to speak about the semi-biproduct of two objects, unless some kind of pseudo-action  is specified. This is no surprise, indeed, in groups, semi-biproducts are the same as semi-direct products with a factor system.  Connections with homology and cohomology are expected via several notions such as abstract kernel, obstruction, etc. (see for example \cite{NMF et all} and references there), 
 %and it is clear that the notion of
 and perhaps the notion of 
  pseudo-action considered here (Definition \ref{def: pseudo-action}) can be used as a tool and clarify some classical  interpretations in group-cohomology.

This paper is organized as follows. In section \ref{sec:def and thms}  we give the definition of semi-biproduct and pseudo-action and establish a categorical equivalence between them (Theorem \ref{thm: equivalence}). In Section \ref{sec: eg}  we give examples of semi-biproducts of monoids obtained from Schreier extensions and other types of more general extensions. In Section \ref{sec:ext} we explore the fact that a semi-biproduct may be viewed as an extension of monoids with a specified section map and a specified retraction map.
%The natural question of where to define semi-biproducts in general is partially answered in the preprint \cite{NMF.20a}, which was the genesis of the current paper.

\section{Definitions and theorems}\label{sec:def and thms}

In this section, we introduce a new notion of \emph{pseudo-action} of mo\-noids, which enables a characterization of semi-biproducts of monoids, as defined in the Introduction. Some previous works have been done in a similar direction but the emphasis has always been on the side of (Schreier) extensions and their classification (e.g. \cite{ NMF14,Fleischer,Ganci,NMF et all}). Here we consider a semi-biproduct as a mathematical object rather than something which arises from an extension with appropriate choices for a section and a retraction. In a certain sense this work goes in the direction of  \cite{GranJanelidzeSobral}. There, the Schreier condition $(\ref{eq: Sch})$ is still being considered, but it shows that it is possible to work in the wider context of unitary magmas. Another recent study on (weakly) Schreier extensions can be found in \cite{Faul}. There, up to some differences concerning what is structure and what is property, the main focus is on a certain quotient, while here we have put our emphasis on a certain subset of the cartesian product between the two ends of a semi-biproduct diagram.

\subsection{Semi-biproducts of monoids} 

%In the preprint \cite{NMF.20a} (the genesis of this paper) a general setting  is established in which semi-biproducts can be defined . When interpreted in the case of monoids (with zero-preserving maps), it gives
A \emph{semi-biproduct} of monoids is
 a diagram of the shape 
\begin{equation}
\label{diag: biproduct}
\xymatrix{X\ar@<-.5ex>[r]_{k} & A\ar@<-.5ex>@{->}[l]_{q}\ar@<.5ex>[r]^{p} & B \ar@{->}@<.5ex>[l]^{s}}
\end{equation}
such that $p$, $k$, are monoid homomorphisms, while  $q$ and $s$ are zero-preserving maps, verifying the conditions
\begin{eqnarray}
ps&=&1_B\\
qk&=&1_X\\
kq+sp&=&1_A\\
pk&=&0_{X,B}\\
qs&=&0_{B,X}.
\end{eqnarray}
% It is possible to characterize all semi-biproducts of monoids with fixed $X$ and $B$.
  We will work with  additive notation on $X$ and   multiplicative notation on $B$, but neither is assumed to be commutative.

\begin{theorem}\label{thm: semi-biproducts}
Let $(X,A,B,p,k,q,s)$ be a semi-biproduct of monoids. If we put for every $b,b'\in B$ and $x\in X$, 
\begin{eqnarray}
b\cdot x=q(s(b)+k(x))\\
x^b=q(k(x)+s(b))\\
b\times b'=q(s(b)+s(b'))
\end{eqnarray}
then, for every $a,a'\in A$
\begin{equation}\label{eq: a+a'}
a+a'=k(q(a)+p(a)\cdot q(a')+(p(a)\times p(a')))+sp(a+a')
\end{equation}
and 
\begin{equation}\label{eq: kuv+sv}
a+a'=k(u^v)+s(v)
\end{equation}
with $u=q(a)+p(a)\cdot q(a')+(p(a)\times p(a'))\in X$ and $v=p(a+a')\in B$.

Moreover, the map $\beta\colon{A\to X\times B}$ with $\beta(a)=(q(a),p(a))$ is always injective. An element $(y,b)\in X\times B$ is in the image of $\beta$ if and only if  $y=x^b$ for some $x\in X$.

This means that $A$ is in bijection with the set $\{(x^b,b)\mid (x,b)\in X\times B\}$ with the inverse map for $\beta$ being $\alpha(x^b,b)=k(x^b)+s(b)=k(x)+s(b)$. 

Furthermore, for every $a\in A$, $q(a)=q(a)^{p(a)}$.
\end{theorem}
\begin{proof}
We observe:
\begin{eqnarray*}
a+a'&=& kqa+(spa+kqa')+spa'  \qquad(kq+sp=1)\\
    &=& kqa+kq(spa+kqa')+sp(spa+kqa')+spa'  \\
    &=& kqa+kq(spa+kqa')+spa+spa'  \qquad( ps=1,pk=0)\\
        &=&  kqa+kq(spa+kqa')+kq(spa+spa')+sp(spa+spa')\\
    &=&  kqa+kq(spa+kqa')+kq(spa+spa')+s(pa+pa')\\
       &=& k(qa+q(spa+kqa')+q(spa+spa'))+sp(a+a') \\
       &=& k(q(a)+p(a)\cdot q(a')+(p(a)\times p(a')))+sp(a+a').
\end{eqnarray*}
We observe further that for every $a\in A$, $x\in X$, $b\in B$, if $a=k(x)+s(b)$ then $a=k(x^b)+s(b)$, which proves condition $(\ref{eq: kuv+sv})$. This follows from the fact that for every $a\in A$, we have $q(a)=q(a)^{p(a)}$. Indeed, $q(a)=q(kq(a)+sp(a))$ follows directly from $a=kq(a)+sp(a)$.

If $\beta(a)=\beta(a')$ then $q(a)=q(a')$ and $p(a)=p(a')$, hence $a=a'$. This proves $\beta$ is injective. 

If $(x,b)\in X\times B$ is of the form $x=q(a)$, $b=p(a)$ for some $a\in A$, then $x=x^b$ because $q(a)=q(a)^{p(a)}$ as already shown. If $(y,b)\in X\times B$ is of the form $y=x^b$ for some $x\in X$, then there exists $a=k(x)+s(b)\in A$ such that $\beta(a)=(x,b)$. In order to see that the map $\alpha(x^b,b)=k(x)+s(b)$ is the inverse of $\beta$ we observe:
\[(x^b,b)\mapsto k(x)+s(b)\mapsto (q(k(x)+s(b)),p(k(x)+s(b))=(x^b,b)\]
and
\[a\mapsto (q(a),p(a))=(q(a)^{p(a)},p(a))\mapsto kq(a)+sp(a)=a.\]
\end{proof} 

As expected, a morphism between semi-biproducts is a diagram of the shape
\begin{equation}\label{diag:morphism of semi-biproduct}
\xymatrix{X\ar@<-.5ex>[r]_{k}\ar@{->}[d]_{f_1} & A\ar@{->}@<0ex>[d]^{f_2}\ar@<-.5ex>@{-->}[l]_{q}\ar@<.5ex>[r]^{p} & B\ar@{->}[d]^{f_3} \ar@{->}@<.5ex>@{-->}[l]^{s}\\
X'\ar@<-.5ex>[r]_{k'} & A' \ar@<-.5ex>@{-->}[l]_{q'}\ar@<.5ex>[r]^{p'} & B' \ar@{->}@<.5ex>@{-->}[l]^{s'}}
\end{equation}
%X\ar@<-.5ex>[r]_(.4){\langle 1,0 \rangle} & X\rtimes_{\varphi,\rho,\gamma}B\ar@{->}@<.5ex>[u]^{\alpha} \ar@<-.5ex>@{-->}[l]_(.6){\pi_{1}}\ar@<.5ex>[r]^(.6){\pi_2} & B \ar@{->}@<.5ex>@{-->}[l]^(.4){{\langle 0 ,1 \rangle}}
in which the top and bottom rows are semi-biproducts, each $f_i$, $i=1,2,3$ is a homomorphism and the following compatibility conditions hold: $f_2k=k'f_1$, $p'f_2=f_3p$, $f_2s=s'f_3$, $q'f_2=f_1q$.

We will show that there is an equivalence of categories between the category of semi-biproducts of monoids and a suitable category of pseudo-actions of monoids.

\subsection{Pseudo-actions of monoids} Let $X$ and $B$ be two monoids. For convenience, let us again use additive notation for $X$ and multiplicative notation for $B$. Note that neither one is assumed to be commutative.

\begin{definition}\label{def: pseudo-action}
A pseudo-action of $B$ on $X$ consists in three different components:
\begin{enumerate}
\item a correction system, that is, a map
\begin{eqnarray*}
\rho\colon{X\times B\to X};\quad 
(x,b)\mapsto x^b
\end{eqnarray*}
such that for all $x\in X$ and $b\in B$
\begin{eqnarray}
x^1=x,\quad 0^b=0,
%(x+y)^b=(x+y^b)^b 
\end{eqnarray}

\item a pre-action, that is, a map
\[\varphi\colon{B\times X\to X};\quad (b,x)\mapsto b\cdot x\] such that for all $b\in B$ and $x\in X$
\begin{eqnarray}
1\cdot x=x,\quad b\cdot 0=0,
\end{eqnarray}

\item a factor system, that is, a map
\begin{eqnarray*}
\gamma\colon{B\times B\to X};\quad 
(b,b')\mapsto b\times b'
\end{eqnarray*}
such that for every $b\in B$
\begin{eqnarray}
1\times b=0=b\times 1.
\end{eqnarray}
\end{enumerate}
The three components are related via one major condition, $(\ref{eq: factor})$, which must hold for every $x,x',x''\in X$ and $b,b',b''\in B$,
\begin{eqnarray}
(x+b\cdot((x'+b'\cdot x'' + (b'\times b''))^{b'b''})+(b\times b'b''))^{bb'b''}=\nonumber\\
=((x+b\cdot x'+(b\times b'))^{bb'}+bb'\cdot x''+(bb'\times b''))^{bb'b''}.\label{eq: factor}
\end{eqnarray}
Moreover, the following two compatibility conditions 
\begin{eqnarray}
(b\cdot x)^b=b\cdot x\label{eq: b.x^b}\\
(b\times b')^{bb'}=b\times b'\label{eq: bxb'^bb'}
\end{eqnarray}
are required for every $x\in X$ and $b,b'\in B$.
\end{definition}

The correction system $\rho$ is used to correct the fact that in general $b\cdot(x+y)$ is not  equal to $b\cdot x+b\cdot y$. Instead, we have the equalities
 \begin{eqnarray}
(b\cdot(x+y))^b=(b\cdot x+b\cdot y)^b\label{eq: correction1}\\
(x+y)^b=(x+y^b)^b \label{eq: correction2}\\
(x^{b}+b\cdot y)^{b}=(x+(b\cdot y)^b)^b \label{eq: correction3}
\end{eqnarray}
which are obtained (Remark \ref{remark} ---  items \ref{item4}, \ref{item5}, \ref{item6} below) as particular cases of $(\ref{eq: factor})$. The fact that the correction system is invisible in groups (Remark \ref{remark} --- item \ref{item2} below) is perhaps the reason why (to the author's best knowledge) it has not appeared before as an explicit structure.  The idea of having one major condition (\ref{eq: factor}) accounting for associativity has been  implicitly used by Leech \cite{Leech} and Wells \cite{Wells} (see also \cite{Fleischer}). The two conditions (\ref{eq: b.x^b}) and (\ref{eq: bxb'^bb'}) are needed in order to have an  equivalence  with semi-biproducts (Theorem \ref{thm: equivalence}).

\begin{remark}\label{remark}
The following particular cases of $(\ref{eq: factor})$ are of interest:
\begin{enumerate}
\item Taking $x,x',x''$ to be zero, we get
\begin{equation}\label{eq: factsystem}
(b\cdot (b'\times b'')^{bb'}+(b\times b'b''))^{bb'b''}=
((b\times b')^{bb'}+(bb'\times b''))^{bb'b''}
\end{equation}
which, if we ignore the correction system, becomes the usual formula for a factor system of monoids (see e.g. \cite{Ganci} and references therein).

\item\label{item2} Taking $x=x''=0$ and $b=1$, we get (by putting again $x$ in the place of $x'$, $b$ in the place of $b'$ and $b'$ in the place of $b''$ for readability)
\begin{equation}\label{eq: 53}
(x^{b}+b\times b')^{bb'}=(x+b\times b')^{bb'}
\end{equation}
which explains how different the correction system is from being an action. In particular, if $X$ is right cancellable and $B$ is a  group then the correction system is always trivial. Indeed, if we take $b'=b^{-1}$ we obtain $x^b+(b\times b^{-1})=x + (b\times b^{-1})$ and hence, cancelling out $b\times b^{-1}$, we get $x^b=x$.

\item Taking $x=x'=0$ and $b''=1$, we get
\begin{equation}\label{eq: factsystemconj}
(b\cdot(b'\cdot x'')^{b'}+(b\times b'))^{bb'}=((b\times b')^{bb'}+bb'\cdot x'')^{bb'}
\end{equation} 
which, in groups, becomes the familiar expression
\begin{equation}
b\cdot(b'\cdot x'')=(b\times b')+bb'\cdot x''-(b\times b')
\end{equation} 
stating that the factor system $b\times b'$ measures, by conjugation, the distance between a pre-action $\varphi$ and an ordinary action.

\item\label{item4} Taking $b=b'=1$ and $x''=0$, we get
\begin{equation}
(x+x')^{b''}=(x+x'^{b''})^{b''}
\end{equation}
which is exactly the same as $(\ref{eq: correction2})$.

\item\label{item5} Taking $b=b''=1$ and $x'=0$, we get
\begin{equation}
(x^{b'}+b'\cdot x'')^{b'}=(x+(b'\cdot x'')^{b'})^{b'}
\end{equation}
which is exactly the same as $(\ref{eq: correction3})$.

\item\label{item6} Taking $b'=b''=1$ and $x=0$, we get
\begin{equation}
(b\cdot (x'+x''))^b=((b\cdot x')^b+b\cdot x'')^b
\end{equation}
which, if combined with $(\ref{eq: correction2})$ and $(\ref{eq: correction3})$, gives $(\ref{eq: correction1})$.

\end{enumerate}
\end{remark}

Given a pseudo-action $(X,B,\rho,\varphi,\gamma)$  we construct (inspired by \cite{Wells}) a \emph{synthetic diagram} 
\begin{equation}\label{diag:synthetic semi-biproduct}
\xymatrix{X\ar@<-.5ex>[r]_(.4){\langle 1,0\rangle} & R_{\rho,\varphi,\gamma}\ar@<-.5ex>@{->}[l]_(.55){\pi_X}\ar@<.5ex>[r]^(.6){\pi_B} & B \ar@{->}@<.5ex>[l]^(.35){\langle 0,1\rangle}}
\end{equation}
with $R_{\rho,\varphi,\gamma}=\{(x,b)\in X\times B\mid x^b=x\}$, equipped with the restriction to the binary operation
\begin{equation}\label{eq: semibiproduct sunthetic operation}
(x,b)+(x',b')=((x+b\cdot x'+(b\times b'))^{bb'},bb')
\end{equation}
defined for every $x,x'\in X$ and $b,b'\in B$.

Let $(X,B,\rho,\varphi,\gamma)$ and $(X',B',\rho',\varphi',\gamma')$ be two pseudo-actions of monoids. A morphism of pseudo-actions is a pair $(f,g)$ of monoid homomorphisms, $f\colon{X\to X'}$, $g\colon{B\to B'}$ preserving each one of the three components of a pseudo-action,
\begin{eqnarray}
f(x^b)&=&f(x)^{g(b)},\\
f(b\cdot x)&=&g(b)\cdot {f(x)},\\
f(b\times b')&=&g(b)\times g(b').
\end{eqnarray}
The equation
\begin{equation}\label{eq: homomprhism conditon}
f((x+b\cdot x'+b\times b')^{bb'})=(f(x)+g(b)\cdot f(x')+ g(b)\times g(b'))^{g(bb')}
\end{equation}
 is a consequence of the previous ones and will be needed in proving that the map ${R_{\rho,\varphi,\gamma}\to R_{\rho',\varphi',\gamma'}}$, defined for every pair $(x,b)\in X\times B$, with $x^b=x$, as $(x,b)\mapsto (f(x),g(b))$, is a monoid homomorphism.

\subsection{The equivalence}

In this subsection we will work towards an equivalence between pseudo-actions and semi-biproducts of monoids.
First we show that there is a functor from the category of semi-bipro\-ducts into the category of pseudo-actions.

\begin{theorem}\label{thm: pseudo-actions}
Let $(X,A,B,p,k,q,s)$ be a semi-biproduct of monoids. The system with  three components
\begin{eqnarray}
x^b=q(k(x)+s(b))\\
b\cdot x=q(s(b)+k(x))\\
b\times b'=q(s(b)+s(b'))
\end{eqnarray}
is a pseudo-action from $B$ into $X$.

Moreover, if $(f_1,f_2,f_3)$ is a morphism of semi-biproducts such as the one displayed in (\ref{diag:morphism of semi-biproduct}) then the pair $(f_1,f_3)$ is a morphism of pseudo-actions.
\end{theorem}

\begin{proof}
Clearly, the formulas given for $x^b$, $b\cdot x$ and $b\times b'$ are well defined maps $\rho$, $\varphi$ and $\gamma$ consistent with the three components of a pseudo-action, respectively a correction system, a pre-action and a factor system. The identities $x^1=x$, $0^b=0$, $1\cdot x=x$, $b\cdot 0=0$, $b\times 1=0=1\times b$ are immediate consequences of $qs=0_{X,B}$, $qk=1_X$, and the fact that the maps $s$ and $q$ preserve the neutral element.

In order to prove condition (\ref{eq: factor}) we observe that for every $x,x',x''\in X$ and $b,b',b''\in B$, the expression
\[kx+sb+kx'+sb'+kx''+sb''\in A\] 
can be decomposed as
\begin{equation}\label{eq:associative1}
\left(\left(kx+\left(\left(sb+kx'\right)+sb'\right)\right)+kx''\right)+sb''
\end{equation}
and as
\begin{equation}\label{eq:associative2}
kx+\left(sb+\left(kx'+\left(\left(sb'+kx''\right)+sb''\right)\right)\right).
\end{equation}
Using the formulas
\begin{eqnarray}
kx+sb=k(x^b)+s(b)\label{eq:kxsb1}\\
sb+kx=k(b\cdot x)+s(b)\label{eq:kxsb2}\\
sb+sb'=k(b\times b')+s(bb')\label{eq:kxsb3}
\end{eqnarray}
when appropriate and following the order of parenthesis indicated in (\ref{eq:associative1}) we obtain the right hand side of (\ref{eq: factor}). Following the order of parenthesis indicated in (\ref{eq:associative2}) we obtain the left hand side of (\ref{eq: factor}).

The condition $(x^b)^b=x^b$ (which is also a consequence of (\ref{eq: factor})) as well as the two conditions (\ref{eq: b.x^b}) and (\ref{eq: bxb'^bb'}) are obtained by applying the map $q$ on both sides of equations (\ref{eq:kxsb1})--(\ref{eq:kxsb3}).

Let us now suppose $(f_1,f_2,f_3)$ is a morphism of semi-biproducts such as the one displayed in (\ref{diag:morphism of semi-biproduct}). We have to show that the pair $(f_1,f_3)$ is a morphism of pseudo-actions.  To simplify notation let us consider $f=f_1$, $g=f_3$ and $h=f_2$. We observe,
\begin{eqnarray*}
f(x)^{g(b)}&=&q'(k'f(x)+s'g(b))\\
&=&q'(hk(x)+hs(b))\\
&=&q'h(k(x)+s(b))\\
&=&fq(k(x)+s(b))\\
&=&f(x^b).
\end{eqnarray*}
Similarly, we prove
\[f(b\cdot x)=fq(sb+kx)=q'h(sb+kx)=q'(sg(b)+k'f(x))=g(b)\cdot f(x)\]
and
\[f(b\times b')=fq(sb+sb')=q'h(sb+sb')=q'(s'g(b)+s'g(b'))=g(b)\times g(b').\]
\end{proof}

%\subsection{Pseudo-actions and semi-biproducts of monoids} 

The previous result gives a functor from the category of semi-bipro\-ducts into the category of pseudo-actions. The synthetic construction (\ref{diag:synthetic semi-biproduct}) produces a functor in the other direction.

\begin{theorem}\label{thm: pseudo-actions to semi-biproducts}
Let $(X,B,\rho,\varphi,\gamma)$ be a pseudo-action of monoids, the synthetic diagram (\ref{diag:synthetic semi-biproduct}) is a semi-biproduct of monoids.
Moreover, if $$(f,g)\colon{(X,B,\rho,\varphi,\gamma)\to (X',B',\rho',\varphi',\gamma')}$$ is a morphism of pseudo-actions then the triple $(f,h,g)$, with  $$h\colon{R_{\rho,\varphi,\gamma}\to R_{\rho',\varphi',\gamma'}},$$ defined as $h(x,b)=(f(x),g(b))$, is a morphism of semi-biproducts.
\end{theorem}

\begin{proof}
 Given a pseudo-action $(X,B,\rho,\varphi,\gamma)$  we consider the synthetic diagram (\ref{diag:synthetic semi-biproduct}). The operation (\ref{eq: semibiproduct sunthetic operation}) is well defined on the set 
 \[R_{\rho,\varphi,\gamma}=\left\lbrace\left(x^b,b\right)\mid x\in X,b\in B\right\rbrace\]
and it is associative by (\ref{eq: factor}). With neutral element $(1,0)$ we see that the set $R=R_{\rho,\varphi,\gamma}$ has a monoid structure. Clearly, the maps $\langle 1,0\rangle\colon{X\to R}$ and $\pi_B\colon{R\to B}$ are monoid homomorphisms while $\langle 0,1\rangle\colon{B\to R}$ and $\pi_X\colon{R\to X}$ are zero-preserving maps. Furthermore, we have $\pi_B\langle 0,1\rangle=1_B$, $\pi_X\langle 1,0\rangle=1_X$, $\pi_B\langle 1,0\rangle=0$ and $\pi_X\langle 0,1\rangle=0$. Finally, we observe that for every $(x,b)\in X\times B$, from (\ref{eq: semibiproduct sunthetic operation}) we have
\[(x,b)=(x,1)+(0,b)=(x^b,b)\] which means that if $(x,b)\in R$ then $(x,b)=(x,1)+(0,b)$. This shows that the synthetic construction is a semi-biproduct.

From equation (\ref{eq: homomprhism conditon}) we see that the map $h(x,b)=(f(x),g(b))$ is a monoid homomorphism. Since it is compatible with the respective semi-biproduct synthetic structure maps, we conclude that it is a morphism of semi-biproducts.
\end{proof}

In order to prove an equivalence of categories between the category of semi-biproducts  and the category of pseudo-actions, we observe that if starting with a structure $(x^b,b\cdot x,b\times b')$ for a pseudo-action of $X$ on $B$, constructing its associated synthetic semi-biproduct and then extracting a new pseudo-action out of it, we obtain the structure $$(x^b,(b\cdot x)^b,(b\times b')^{bb'}).$$ The role of equations (\ref{eq: b.x^b}) and (\ref{eq: bxb'^bb'}) is to ensure that the new structure is in fact equal to the original one. On the other hand, if starting with a semi-biproduct of monoids, extracting its associated pseudo-action, and then reconstructing the synthetic semi-biproduct, we get a natural isomorphism of semi-biproducts.

\begin{theorem}\label{thm: equivalence}
There is an equivalence between the category of semi-biproducts of monoids and the category of pseudo-actions.
\end{theorem}

\begin{proof}
Let us start with a semi-biproduct $(X,A,B,p,k,q,s)$ and let $(X,B,\rho,\varphi,\gamma)$ be its associated pseudo-action. The synthetic semi-biproduct constructed from the pseudo-action is isomorphic to the original semi-biproduct as illustrated
\begin{equation}
%\label{diag:semi-biproduct}
\xymatrix{X\ar@<-.5ex>[r]_{k}\ar@{=}[d] & A\ar@{->}@<.5ex>[d]^{\beta}\ar@<-.5ex>@{-->}[l]_{q}\ar@<.5ex>[r]^{p} & B\ar@{=}[d] \ar@{->}@<.5ex>@{-->}[l]^{s}\\
X\ar@<-.5ex>[r]_(.4){\langle 1,0 \rangle} & R_{\rho,\varphi,\gamma}\ar@{->}@<.5ex>[u]^{\alpha} \ar@<-.5ex>@{-->}[l]_(.6){\pi_{X}}\ar@<.5ex>[r]^(.6){\pi_B} & B \ar@{->}@<.5ex>@{-->}[l]^(.4){{\langle 0 ,1 \rangle}}}
\end{equation}
with the morphisms  $\alpha$ and $\beta$ defined as in Theorem \ref{thm: semi-biproducts}. We have to show that $\alpha$ and $\beta$ are monoid homomorphisms. It is immediate to see that they are compatible with the semi-biproduct structuring maps, that is
\begin{eqnarray*}
p\alpha(x,b)=p(kx+sb)=b,\\
q\alpha(x,b)=q(kx+sb)=x^b=x\quad, (x,b)\in R\\
\alpha(0,b)=s(b),\quad \alpha(x,1)=k(x)\\
\pi_B\beta(a)=p(a),\quad \pi_X\beta(a)=q(a)\\
\beta s(b)=(0,b),\quad \beta k(x)=(x,1).
\end{eqnarray*} 
In order to prove that $\alpha$ is a homomorphism we observe
\[\alpha((x,b)+(x',b'))=k((x+b\cdot x'+b\times b')^{bb'})+s(bb')\]
and
\begin{eqnarray*}
\alpha(x,b)+\alpha(x',b')&=&kx+sb+kx'+sb'\\
&=&kx+k(b\cdot x')+sb+sb'\\
&=&kx+k(b\cdot x')+k(b\times b')+s(bb')\\
&=&k(x+b\cdot x'+b\times b')+s(bb')\\
&=&k((x+b\cdot x'+b\times b')^{bb'})+s(bb').
\end{eqnarray*}
In order to prove that $\beta$ is a homomorphism we observe that $\beta(a+a')=(q(a+a'),p(a+a'))$ while $\beta(a)+\beta(a')=(q(a),p(a))+(q(a'),p(a'))$. This means that $\beta$ is a homomorphism as soon as
\begin{equation}
q(a+a')=(q(a)+p(a)\cdot q(a')+(p(a)\times p(a')))^{p(a+a')}
\end{equation}
which follows from (\ref{eq: a+a'}).

To see that $\alpha$ and $\beta$ define a natural isomorphism between a semi-biproduct and its associated synthetic one, let us consider the diagram
\begin{equation}
\xymatrix{A\ar[r]^{f_2}\ar@<.5ex>[d]^{\beta} & A'\ar@<.5ex>[d]^{\beta'}\\R\ar@<.5ex>[u]^{\alpha}\ar[r]^{h}&R' \ar@<.5ex>[u]^{\alpha'}}
\end{equation}
in which $f_2$ is part of a morphism $(f_1,f_2,f_3)$ between semi-biproducts and $h$ is the  induced morphism between respective associated synthetic semi-biproducts, that is, $h(x,b)=(f_1(x),f_3(b))$. We observe
\[h\beta(a)=(f_1q(a),f_3p(a))=(q'f_2(a),p'f_2(a))=\beta'f_2(a)\]
and
\begin{eqnarray*}
f_2\alpha(x,b)&=&f_2(kx+sb)=f_2k(x)+f_2s(b)\\
&=&k'f_1(s)+s'f_3(b)=\alpha'h(x,b).
\end{eqnarray*}

If we start with a pseudo-action $(X,B,\varphi,\rho,\gamma)$, build its associated synthetic semi-biproduct, and extract the pseudo-action associated to it, then a new pseudo-action is obtained, let us call it $(X,B,\rho',\varphi',\gamma')$. The new pseudo-action is $\rho'(x,b)=x^b$, $\varphi'(b,x)=(b\cdot x)^b$ and $\gamma'(b,b')=(b\times b')^{bb'}$.  The two conditions (\ref{eq: b.x^b}) and (\ref{eq: bxb'^bb'}) ensure that the new pseudo-action and the original one are the same.
\end{proof}

\section{Examples}\label{sec: eg}

Every Schreier split extension (\cite{NMF14}) of a monoid $X$ by a monoid $B$  is a semi-biproduct of monoids. In this case $b\times b'=0$ and $x^b=x$ for all $x\in X$, $b,b'\in B$. More generally, every Schreier extension of monoids is a semi-biproduct in which $x^b=x$. In fact, a semi-biproduct of monoids $(X,A,B,p,k,q,s)$ is a Schreier extension if and only if $x^b=x$ for all $x\in X$ and $b\in B$.

Here is a general procedure for the construction of semi-biproducts which are not necessarily Schreier extensions. The procedure is obtained from a pseudo-action $(X,B,\rho,\varphi,\gamma)$ by isolating the subset $$\{(\rho(x,b),b)\mid x\in X,b\in B\}\subseteq X\times B$$ and then working backwards in order to recover the rest of the structure while assuming that the map $s$ is of the form $s(b)=(u(b),b)$ for some map $u$ from $B$ to $X$.

Let $X$ (written additively) and $B$ (written multiplicatively) be two monoids and let us suppose the existence of a subset $R\subseteq X\times B$, considered as a binary relation, so that we write $xRb$ in the place of $(x,b)\in R$, together with two maps $u\colon{B\to X}$ and $q\colon{R\to X}$ satisfying the following conditions:
\begin{eqnarray}
\text{ for all $x\in X$, $xR1$ and $q(x,1)=x$},\\
\text{ for all $b\in B$, $u(b)Rb$ and $q(u(b),b)=0$},\\
\text{ if $xRb$, $yRb$ and $q(x,b)=q(y,b)$ then $x=y$.}\label{item:q}
\end{eqnarray}
For every monoid structure on the set $R$, for which $(0,1)$ is the neutral element and the projection map $(x,b)\mapsto b$ is a homomorphism, we put:
\begin{eqnarray*}
x\oplus x'&=&q((x,1)+(x',1))\\
b\times b'&=&q((u(b),b)+(u(b'),b'))\\
b\cdot x&=&q((u(b),b)+(x,1))\\
x^b&=&q((x,1)+(u(b),b)).
\end{eqnarray*}
If $x\oplus x'=x+x'$ for all $x,x'\in X$ (where $+$ is the monoid operation on $X$) and if $q(x,b)^b=q(x,b)$, whenever $xRb$, then we have a semi-biproduct
\begin{equation}
%\label{diag:semi-biproduct}
\xymatrix{X\ar@<-.5ex>[r]_{k} & R\ar@<-.5ex>@{-->}[l]_{q}\ar@<.5ex>[r]^{p} & B \ar@{-->}@<.5ex>[l]^{s}}
\end{equation}
with $p(x,b)=b$, $k(x)=(x,1)$, $s(b)=(u(b),b)$.
 The monoid operation on $R$ is given by the formula
\begin{equation}
(x,b)+(x',b')=((x+b\cdot x'+ b\times b')^{bb'},bb').
\end{equation}
Indeed, $k$ is a homomorphism because $\oplus$ is the monoid operation defined on $X$, $p$ is a homomorphism by assumption, $pk=0$, $ps=1$, $qk=1$, $qs=0$ by construction  and we have, for every $xRb$,
\begin{equation}\label{eq:xRb}
(x,b)=(q(x,b),1)+(u(b),b)
\end{equation}
which means that $kq+sp=1_R$. It is not difficult to see that $(\ref{eq:xRb})$ follows from the fact that $q(x,b)^b=q(x,b)$, together with the fact that by condition $(\ref{item:q})$ the map $\langle q,p\rangle\colon{R\to X\times B}$ is injective. 
%Note that a semi-biproduct thus obtained is a Schreier extension if (and only if) $R$ is in bijection with the set $X\times B$.
In particular, when we force $q(x,b)=x$ and $u(b)=0$ then we are reduced to the analysis of monoid structures on $R$. In that case, the condition $q(x,b)^b=q(x,b)$ becomes $x^b=x$ whenever $xRb$, and if we are looking for semi-biproducts which are not Schreier extensions then we have to consider relations $R$ which are not in bijection with the set $X\times B$. Let us work a concrete example.

Take $X=\{0,s\}$ with $s+s=s$ and $B=\{1,t\}$ with $t^2=t$. Clearly, $R=\{(0,1),(s,1),(0,t)\}$ is the only possible proper subset of $X\times B$ with $xR1$ and $0Rb$ for every $x\in X$ and $b\in B$. There are two possible solutions to turn $R$ into a monoid with $0R1$ as neutral element and such that the map $p(x,b)=b$ is a homomorphism.
\begin{enumerate}
\item One possibility is to map $0R1\mapsto 1$, $sR1\mapsto -1$, $0Rt\mapsto 0$ and take the usual multiplication on the set $\{-1,0,1\}$. However, in this case we find that $s\oplus s=0$ while $s+s=s$. So, it does not give a semi-biproduct.

\item Another possibility is to consider the chain semilattice structure on $R$ as follows
\begin{equation*}
\begin{tabular}{c|ccc}
+ & $0R1$ & $sR1$ & $0Rt$\\
\hline
$0R1$ & $0R1$ & $sR1$ & $0Rt$\\
$sR1$ & $sR1$ & $sR1$ & $0Rt$\\
$0Rt$ & $0Rt$ & $0Rt$ & $0Rt$
\end{tabular}
\end{equation*}
In this case we have
\begin{equation*}
\begin{tabular}{c|c|c|c|c|c|c|c}
$x$ & $b$ & $x'$ & $b'$ & $x\oplus x'$ & $b\times b'$ & $b\cdot x$ & $x^b$\\
\hline
$0$ & $1$ & $0$ & $1$ & $0$ & $0$ & $0$ & $0$ \\
$0$ & $t$ & $s$ & $1$ & $s$ & $0$ & $0$ & $0$ \\
$s$ & $1$ & $0$ & $t$ & $s$ & $0$ & $s$ & $s$ \\
$s$ & $t$ & $s$ & $t$ & $s$ & $0$ & $0$ & $0$ 
\end{tabular}
\end{equation*}
and since $x\oplus x'=x+x'$ and $x^b=x$ whenever $xRb$ (note that $s^t=0$ but $(s,t)\notin R$) we obtain a semi-biproduct of monoids 
\begin{equation}
%\label{diag:semi-biproduct}
\xymatrix{\{0,s\}\ar@<-.5ex>[r]_{\iota_1} & R\ar@<-.5ex>@{-->}[l]_{\pi_1}\ar@<.5ex>[r]^{\pi_2} & \{1,t\} \ar@{->}@<.5ex>[l]^{\iota_2}}
\end{equation}
with $\pi_2(x,b)=b$, $\iota_2(b)=(0,b)$, $\pi_1(x,b)=x$, $\iota_1(x)=(x,1)$. It cannot be a Schreier extension because $R$ has three elements while $X\times B$ has four elements. Note that this example is the same as (\ref{eg:simple eg}), from Introduction.
\end{enumerate}

If we now consider the case when $X=\{0,s\}$ with $s+s=0$, while keeping $B=\{1,t\}$ with $t^2=t$, then we observe that the same mapping $0R1\mapsto 1$, $sR1\mapsto -1$, $0Rt\mapsto 0$ as before, with the usual multiplication on the set $\{-1,0,1\}$, gives a semi-biproduct
\begin{equation}
%\label{diag:semi-biproduct}
\xymatrix{X\ar@<-.5ex>[r]_(.35){k} & \{1,-1,0\}\ar@<-.5ex>@{-->}[l]_(.65){q}\ar@<.5ex>[r]^(.65){p} & B \ar@{-->}@<.5ex>[l]^(.35){s}}
\end{equation}
if we let $X\cong\{1,-1\}$ and $B\cong\{1,0\}$ have the structure of the usual multiplication while $p,k,q,s$ are obvious inclusions and retractions. Thus we see that even when $X$ is a group (but not $B$), we can have an extension which is not a Schreier extension.

In the case when $B$ is a group, say $B=\{1,b\}$ with $b^2=1$, we may still produce one example of a semi-biproduct  as soon as we consider $X=\{1,a\}$ and $A=\{1,a,b\}$ (compare with diagram (\ref{eg:simple eg})) with $a^2=b^2=a$ and $ab=ba=b$ (note that in this case $X$ does not admit cancellation, otherwise every semi-biproduct would be a Schreier extension as observed in  Remark~\ref{remark}, item \ref{item2}.

A simple example that illustrates the case when  $R$ is in bijection with $X\times B$, that is when $x^b=x$, is the following. Put $X=B=\mathbb{N}$ the additive monoid of natural numbers and consider the order relation $R=\{(x,b)\in \mathbb{N}^2\mid x\geq b\}$ together with the two maps $q(x\geq b)=x-b$ and $u(b)=b$. The usual component-wise addition on $R$ verifies all the conditions specified by the  construction scheme outlined above and hence it gives rise to a semi-biproduct. Namely, $(\mathbb{N},R,\mathbb{N},p,k,q,s)$ with $p(x\geq b)=b$, $k(x)=(x\geq 0)$, $q(x\geq b)=x-b$ and $s(b)=(b\geq b)$. In this case $q$ and $s$ are both homomorphisms and hence $x^b=x$ for every $(x,y)\in \mathbb{N}\times \mathbb{N}$ so that we obtain an isomorphism
\[
\xymatrix{\mathbb{N}\ar@{=}[d]_{}\ar@<-.5ex>[r]_{k} & R\ar@<-.5ex>[d]_{\beta}\ar@<-.5ex>@{->}[l]_{q}\ar@<.5ex>[r]^{p} & \mathbb{N}\ar@{=}[d]^{} \ar@{->}@<.5ex>[l]^{s}\\
\mathbb{N}\ar@<-.5ex>[r]_{\langle 1,0\rangle} & \mathbb{N}\times \mathbb{N}\ar@<-.5ex>[u]_{\alpha}\ar@<-.5ex>@{->}[l]_{\pi_1}\ar@<.5ex>[r]^{\pi_2} & \mathbb{N} \ar@{->}@<.5ex>[l]^{\langle 0,1\rangle}}
\]
with $\beta(x\geq b)=(x-b,b)$ and $\alpha(x,b)=(x+b\geq b)$. So we see that the usual order relation on the natural numbers is part of a semi-biproduct of monoids whose associated synthetic semi-biproduct is the familiar biproduct on the natural numbers.

\section{The underlying extension}\label{sec:ext}

 In this section we observe some slight adaptations of well-known results suggesting that semi-biproducts can be studied in the light of $S$-protomodular categories \cite{DB.NMF.AM.MS.16}.
A semi-biproduct $(X,A,B,p,k,q,s)$ can be seen as  an extension of monoids $X\to A\to B$ equipped with a section $s$ and a retraction $q$. Moreover, such extensions are pullback stable and the pair $(k,s)$ is jointly epimorphic in the sense that every monoid homomorphism $g\colon{A\to Y}$, from $A$ into a monoid $Y$, is completely determined by the homomorphism $gk\colon{X\to Y}$ and the map $gs\colon{B\to Y}$, as $g(a)=gkq(a)+gsp(a)$, for every $a\in A$. Similarly, the pair $(q,p)$ is jointly monomorphic in the sense that a homomorphism $f\colon{Z\to A}$, from a monoid $Z$ into the monoid $A$, is completely determined by the homomorphism $pf\colon{Z\to B}$ and the map $qf\colon{Z\to X}$, as $f(z)=kqf(z)+spf(z)$, for every $z\in Z$.

\begin{proposition}\label{thm:kernel of p}
Let $(X,A,B,p,k,q,s)$ be a semi-biproduct of monoids. The morphism $k$ is the kernel of the morphism $p$.
\end{proposition} 
\begin{proof}
Let $f\colon{Z\to A}$ be a morphism such that $pf=0$. Then the map $\bar{f}=qf$ is a homomorphism
\begin{eqnarray*}
qf(z+z')&=&q(fz+fz')=q(kqf(z)+spf(z)+kqf(z')+spf(z'))\\
&=& q(kqf(z)+0+kqf(z')+0)\\
&=& qk(qf(z)+qf(z'))=qf(z)+qf(z')
\end{eqnarray*}
and it is unique with the property $k\bar{f}=f$. Indeed, if $k\bar{f}=f$ then $qk\bar{f}=qf$ and hence $\bar{f}=qf$.
\end{proof}

\begin{proposition}\label{thm:cokernel of k}
Let $(X,A,B,p,k,q,s)$ be a semi-biproduct of monoids. The morphism $p$ is the cokernel of the morphism $k$.
\end{proposition}
\begin{proof}
Let $g\colon{A\to Y}$ be a morphism and suppose that $gk=0$. It follows that $g=gsp$,
\begin{equation*}
g=g1_A=g(kq+sp)=gkq+gsp=0+gsp=gsp,
\end{equation*}
and consequently the map $\bar{g}=gs$ is a homomorphism,
\begin{eqnarray*}
gs(bb')=gs(ps(b)ps(b'))=gsp(sb+sb')=g(sb+sb')=gs(b)+gs(b').
\end{eqnarray*}
The fact that $\bar{g}=gs$ is the unique morphism with the property $\bar{g}p=g$ follows from $\bar{g}ps=gs$ which is the same as $\bar{g}=gs$.
\end{proof}

\begin{proposition}\label{thm:stable under pullback}Semi-biproducts of monoids are stable under pullback.
\end{proposition}
\begin{proof}
Let $(X,A,B,p,k,q,s)$ be a semi-biproduct of monoids displayed as the bottom row in the following diagram which is obtained by taking the pullback of $p$ along an arbitrary morphism $h\colon{C\to B}$, with induced morphism $\langle k,0 \rangle$ and map $\langle sh,1 \rangle$,
\begin{eqnarray}
\xymatrix{X\ar@<-.5ex>[r]_(.35){\langle k,0 \rangle}\ar@{=}[d]_{} & A\times_B C\ar@{->}@<0ex>[d]^{\pi_1}\ar@<-.5ex>@{-->}[l]_(.6){q\pi_1}\ar@<.5ex>[r]^(.6){\pi_2} & C\ar@{->}[d]^{h} \ar@{->}@<.5ex>@{-->}[l]^(.35){\langle sh,1\rangle}\\
X\ar@<-.5ex>[r]_{k} & A \ar@<-.5ex>@{-->}[l]_{q}\ar@<.5ex>[r]^{p} & B \ar@{->}@<.5ex>@{-->}[l]^{s}.}
\end{eqnarray}
We have to show that the top row is a semi-biproduct of monois. By construction we have $\pi_2\langle sh,1\rangle=1_C$, $\pi_2\langle k,0\rangle=0$, $q\pi_1\langle sh,1\rangle=qsh=0$, $q\pi_1\langle k,0\rangle=qk=1_X$. It remains to prove the identity
\[(a,c)=(kq(a),0)+(sh(c),c)=(kq(a)+sh(c),c)\]
for every $a\in A$ and $c\in C$ with $p(a)=h(c)$, which follows from $a=kq(a)+sp(a)=kq(a)+sh(c)$.  
\end{proof}

The following list of examples illustrates some aspects of semi-biproducts when considered as an extension with specified section and retraction. Let us consider eight different diagrams displayed as
\begin{equation}
%\label{diag:semi-biproduct}
\xymatrix{X\ar@<-.5ex>[r]_{k} & A_i \ar@<-.5ex>@{-->}[l]_{q}\ar@<.5ex>[r]^{p} & B \ar@{-->}@<.5ex>[l]^{s}},\quad i=1,\ldots,8.
\end{equation}
Each monoid $A_i$ has underlying set $\{0,a,b,c,d\}$ and respective operation defined by the table
 \begin{eqnarray*}
 A_1=\left(\begin{array}{ccccc}
0 & a & b & c & d\\ 
a & a & b & c & d\\ 
b & b & b & c & d\\ 
c & c & d & d & d\\ 
d & d & d & d & d
\end{array}\right),
 & A_2=\left(\begin{array}{ccccc}
0 & a & b & c & d\\ 
a & a & b & c & d\\ 
b & b & b & c & d\\ 
c & c & c & d & c\\ 
d & d & d & c & d
\end{array}\right), \\
 A_3=\left(\begin{array}{ccccc}
0 & a & b & c & d\\ 
a & a & b & c & c\\ 
b & b & b & c & c\\ 
c & c & c & c & c\\ 
d & d & c & c & c
\end{array}\right),
 & A_4=\left(\begin{array}{ccccc}
0 & a & b & c & d\\ 
a & a & b & c & c\\ 
b & b & b & c & c\\ 
c & c & c & c & c\\ 
d & c & c & c & c
\end{array}\right)
 \end{eqnarray*}
and
 \begin{eqnarray*}
 A_5=\left(\begin{array}{ccccc}
0 & a & b & c & d\\ 
a & a & b & d & d\\ 
b & b & b & d & d\\ 
c & c & c & d & d\\ 
d & d & d & d & d
\end{array}\right),
 & A_6=\left(\begin{array}{ccccc}
0 & a & b & c & d\\ 
a & a & b & d & d\\ 
b & b & b & d & d\\ 
c & d & d & d & d\\ 
d & d & d & d & d
\end{array}\right), \\
 A_7=\left(\begin{array}{ccccc}
0 & a & b & c & d\\ 
a & a & b & d & d\\ 
b & b & b & d & d\\ 
c & c & c & c & c\\ 
d & d & d & d & d
\end{array}\right),
 & A_8=\left(\begin{array}{ccccc}
0 & a & b & c & d\\ 
a & a & b & d & d\\ 
b & b & b & d & d\\ 
c & d & d & c & d\\ 
d & d & d & d & d
\end{array}\right). 
 \end{eqnarray*}
The monoid $X$ is the three element linear chain (sup-semi-lattice) whereas the monoid $B$ is the two element linear chain. More specifically we consider $X$ as the set $\{0,a,b\}$ and $B$ as the set $\{0,c\}$ with operation tables
 \begin{eqnarray*}
 X=\left(\begin{array}{ccc}
0 & a & b \\ 
a & a & b  \\ 
b & b & b  
\end{array}\right),
 & B=\left(\begin{array}{cc}
0 & c\\ 
c & c 
\end{array}\right).
 \end{eqnarray*}
The given presentation for the monois $X$, $B$ and $A_i$ permits us to consider $k$ as the inclusion homomorphism, $s$ as the inclusion map, $p$ as the homomorphism defined by  $p^{-1}(0)=\{0,a,b\}$, $p^{-1}(c)=\{c,d\}$ and $q$ as the map defined by $q^{-1}(0)=\{0,c\}$, $q^{-1}(a)=\{a,d\}$, $q^{-1}(b)=\{b\}$, in all the eight cases. 

For each structure $(X,A_i,B,p,k,q,s)$, $i=1,\ldots,8$, we observe:
\begin{enumerate}
\item The cases 1 to 4 are not semi-biproducts because the element $d\in A$ cannot be obtained as $d=kq(d)+sp(d)$, indeed $kq(d)+sp(d)=a+c=c\neq d$.

\item The cases 5 to 8 are semi-biproducts of monoids. Indeed, we have $pk=0$, $ps=1_B$, $qk=1_X$, $qs=0$ and we have $0=kq(0)+sp(0)=0+0$, $a=kq(a)+sp(a)=a+0$, $b=kq(b)+sp(b)=b+0$, $c=kq(c)+sp(c)=0+c$ in all eight cases. However, the condition $d=kq(d)+sp(d)=a+c$ is satisfied in $A_5$ to $A_8$ but not in $A_1$ to $A_4$ as explained in the previous item. 

\item The inclusion map $s$ is a homomorphism in the cases 3, 4, 7 and 8 but not in the cases 1, 2, 5, 6.

\item The map $q$ is never a homomorphism. Indeed, in the cases 1 to 4 we have $q(b+c)=q(c)=0$ whereas $q(b)+q(c)=b+0=b$. In the cases 5 to 8 we have $q(b+c)=q(d)=a$ whereas $q(b)+q(c)=b+0=b$.

\item In the cases 2, 4, 6 and 8 the monoid $A_i$ is commutative whereas in the other cases it is not.

\item In all cases $k$ is the kernel of $p$ but for example in the cases 1 and 2 the pair $(q,p)$ is not jointly monomorphic in the sense above. Indeed, the inclusion homomorphism $f\colon{Z\to A_i}$ with $Z=\{0,d\}$ is not determined by $qf$ and $pf$ since
\[kqf(d)+spf(d)=kq(d)+sp(d)=a+c=c\neq d.\]

\item In all cases except 2 the morphism $p$ is the cokernel of $k$. In the cases 5 to 8 by Proposition \ref{thm:cokernel of k}, in the cases 1, 3 and 4 because $p$ is a split epimorphism (in case 1, $s$ is not a homomorphism but $s'(c)=d$ is). In the case 2 we see that the pair $(k,s)$ is not jointly epimorphic in the sense above. For example, the homomorphism $g\colon{A_2\to Y}$, with $Y=\{0,c,d\}$ obtained as $Y=A_2([1,4,5],[1,4,5])$, and defined by $g^{-1}(0)=\{0,a,b\}$, $g^{-1}(c)=\{c\}$, $g^{-1}(d)=\{d\}$ is not determined by $gk$ and $gs$. Indeed, $g(d)=d$ cannot be obtained as $gkq(d)=0$ and $gsp(d)=c$. In particular we observe that there is no $\bar{g}\colon{B\to Y}$ such that $g=\bar{g}p$.
\end{enumerate}

The previous propositions are simple adaptations of well-known results, especially when $s$ is a homomorphism. Moreover, since monoid monomorphisms are injective maps we observe further that each pair $(k,s)$ in a semi-biproduct is jointly extremal-epimorphic. Indeed, given any monomorphism $m\colon M\to A$ for which there exists a homomorphism $l\colon X\to M$ and a map $t\colon B\to M$ such that $ml=k$ and $mt=s$, we have $a=kq(a)+sp(a)=m(lq(a)+tp(a))$, and so $m$ is an isomorphism. Note that this argument is not dual because not every epimorphism is a surjective map.

\section{Conclusion}

We have introduced the notion of semi-biproduct of monoids. Considerations on categorical aspects related to this notion can be found in \cite{NMF.20a}. Several directions using different techniques are pointed out there. The main issue involves dealing with morphisms and maps on the same ground. One possibility is to consider maps as an extra structure in higher dimensions \cite{Brown,NMF.15}. Another one is to consider maps as imaginary morphisms \cite{BZ,MontoliRodeloLinden}. It seems to be relevant the investigation on both directions and to determine whether they lead to different settings or if rather there is a more general setting which unifies both directions.
Developing a general categorical framework in which to study semi-biproducts is desirable due to several interesting cases that occur in different contexts. For example, semi-biproducts can be studied in the context of preordered monoids (see \cite{NMF.20b}, see also \cite{Preord} for preordered groups) where the maps $q$ and $s$ can be required to be monotone maps rather than merely zero-preserving maps. The context of topological monoids \cite{Ganci} is also worthwhile studying since we can take the maps $q$ and $s$ to be continuous or not and possible variations on that.

\section*{Acknowledgements}
The author would like to thank the referee for a very careful reading of the text, which led to the current improved version; and the Editor for   helpful suggestions on Section~\ref{sec:ext}. 
This work was supported by Fundação para a Ciência e a Tecnologia (FCT UID-Multi-04044-2019), Centro2020 (PAMI -- ROTEIRO\-/0328\-/2013 -- 022158) and by the Polytechnic of Leiria through the projects CENTRO\--01\--0247: FEDER\--069665, FEDER\--069603, FEDER\--039958, FEDER\--039969, FEDER\--039863, FEDER\--024533.

%\bibliography{tim}
%\bibliographystyle{amsplain}

\end{document}